\documentclass[11pt]{article}
\usepackage{amsmath,amsthm,amssymb,amscd,fancybox,ifthen,float,epsfig,subfigure}
\usepackage[all]{xy}
\usepackage{times}

\floatstyle{plain}
\restylefloat{figure}

\newcommand\diag{\operatorname{diag}}

\newcommand\Dir{\operatorname{Dir}}

\newcommand\cH{\mathcal H}

\newcommand\tk{\tilde k}

\newcommand\bn{\boldsymbol n}

\newcommand\hx{\widehat{x}}
\newcommand\hg{\widehat{g}}
\newcommand\hmu{\widehat{\mu}}

\newcommand\Ker{\operatorname{ker}}

\newcommand\cC{\mathcal{C}}
\newcommand\cL{\mathcal{L}}
\newcommand\cS{\mathcal{S}}

\newcommand\cD{\mathcal{D}}

\renewcommand\Im{\operatorname{Im}}

\newcommand\bbN{\mathbb N}

\newcommand\bbR{\mathbb R}

\newcommand\pa{\partial}

\newcommand\restrictedto{\upharpoonright}

\newcommand\Id{\operatorname{Id}}
\newcommand\Sgn{\operatorname{sgn}}

\newtheorem{theorem}{Theorem}

\theoremstyle{definition}

\theoremstyle{remark}
\newtheorem{remark}{Remark}

\begin{document}
\title{On the stability of time-domain integral equations for acoustic wave propagation} 

\author{Charles L. Epstein,\footnote{
    Depts. of Mathematics and Radiology, University of Pennsylvania,
    209 South 33rd Street, Philadelphia, PA 19104. E-mail:
    {cle@math.upenn.edu}.
    Research partially supported by NSF grant
    DMS12-05851 and ARO grant W911NF-12-1-0552.} \,\,
 Leslie Greengard,\footnote{Courant Institute,
    New York University, 251 Mercer Street, New York, NY 10012 and
    Simons Foundation, 160 Fifth Avenue, New York, NY 10010.
    E-mail: {greengard@cims.nyu.edu}. 
    Research partially supported by the U.S. Department of Energy under
    contract DEFG0288ER25053 and by 
    by the Office of the Assistant Secretary of Defense for Research and Engineering 
    and AFOSR under NSSEFF Program Award FA9550-10-1-0180.} \\ 
and Thomas Hagstrom\footnote{Dept. of Mathematics, Southern Methodist University,
    PO Box 750156, Dallas TX 75275-0156. E-mail: {thagstrom@smu.edu}. Research
    partially supported by grants ARO W911NF-09-1-0344 and the NSF DMS-1418871.
    }}  
\date{April 15, 2015}

\maketitle
\centerline{\large In honor of Peter Lax on the occasion of his 90th birthday.}
\begin{abstract}We give a principled approach for the selection of a boundary
  integral, retarded potential representation for the solution of scattering
  problems for the wave equation in an exterior domain.
\end{abstract}

\section{Introduction}
Let $D$ be a bounded domain with smooth boundary, $\Gamma,$ and set
$\Omega=\overline{D}^c.$ The determination of the outgoing part of the solution
to the classical acoustic scattering problem for a soft scatterer requires
solving the mixed Cauchy problem for the wave equation:
\begin{equation}\label{eqn1}
\begin{split}
  &\pa_t^2u=c^2\Delta u\text{ in }\Omega\times [0,\infty),\\
&\text{ with }u(0,x)=\pa_t u(0,x)=0\text{ and }\\
&u(x,t)=g(x,t)\text{ for } (x,t)\in \pa\Omega\times [0,\infty).
\end{split}
\end{equation}
Various other types of scatterers are modeled by using boundary conditions of the
form
$\beta(x)u(x,t)+\pa_{\bn}u(x,t)=g(x,t),$ for $(x,t)\in \pa\Omega\times
[0,\infty).$

Many numerical approaches to solving this problem express the solution in terms
of the retarded potentials derived from the fundamental solution to the wave
equation. By analogy with the time harmonic case, we define the single and double
layer potentials by the formulae
\begin{equation}
  \begin{split}
    \cS\mu(x,t)&=\int\limits_{\Gamma}\frac{\mu(y,t-|x-y|)}{4\pi |x-y|}dS_y\\
\cD\mu(x,t)&=\int\limits_{\Gamma}\frac{\bn_y\cdot(x-y)}{4\pi |x-y|}\left[
\frac{\mu(y,t-|x-y|)}{|x-y|^2}+\frac{\dot{\mu}(y,t-|x-y|)}{|x-y|}\right]dS_y,
  \end{split}
\end{equation}
for $x\in\Gamma^c$.  Here, $\dot{\mu}(y,s)=\pa_s\mu(y,s).$ 

Suppose now that we represent the solution in the form
\begin{equation}\label{eqn4}
  u(x,t)=\cD\mu(x,t)+\cS(a\dot{\mu}+b\mu)(x,t)\text{ for }(x,t)\in\Omega\times [0,\infty).
\end{equation}
To enforce the boundary condition on $\Gamma\times [0,\infty)$, 
we take the limit of this representation 
from the exterior of $D$, yielding the integral equation 
\begin{equation}\label{eqn3}
  \left[\frac{\mu(x,t)}{2}+N\mu(x,t)+K(a(y)\dot{\mu}+b(y)\mu)(x,t)\right]=g(x,t),
\end{equation}
for $(x,t)\in \Gamma\times [0,\infty)$. 
Here $N$ is the weakly singular principal part of the double
layer restricted to $\Gamma\times\Gamma,$ and $K$ is the single layer on
$\Gamma\times\Gamma.$ 

There is a substantial literature on time-domain integral equations for scattering
problems, which we do not seek to review here. 
(See, for example, \cite{COSTABEL,SAYAS,ERGIN,HADUONG,HADUONG2,LEVIATAN} and the references therein.) 
We simply note that, while prior work has considered the use of the single layer 
alone, the double layer alone,
as well as linear combinations of the two, there is relatively little discussion of a
principled analytic approach for selecting one representation over another.
Using classical results from scattering
theory of Lax, Morawetz, and Phillips \cite{LMP63}, we provide criteria for
choosing the coefficients $(a(y)$ and $b(y))$ above, at least when $\Gamma$ is a
non-trapping obstacle. This is a hypothesis about the behavior of reflected
light rays in the exterior of $D,$ which is always satisfied if $D$ is convex
or star-shaped. These results imply that if the disturbance $g(y,t)$ has compact
support in time, then the solution to the mixed Cauchy problem in
equation~\eqref{eqn1} will eventually exhibit local exponential decay. 

The principle that guides our selection of representation is that, ideally, the
source $\mu(y,t)$ should display the same rate of exponential decay in time as
the solution itself. If this is not the case, then the mapping from $\mu\mapsto
u(x,t)$ in~\eqref{eqn4} must, after a certain time, involve some catastrophic
cancellation, and hence can be expected to lose relative accuracy as time
progresses. This goal can be achieved exactly for round spheres. For more
general non-trapping obstacles, we show (modulo a technical hypothesis) that
certain choices of $(a,b)$ lead to sources that decay exponentially, though
perhaps not at the same rate as $u$ itself.  As is shown in Section~\ref{sec5},
in the smooth, strictly convex case, microlocal analysis indicates that the
choice $a\equiv 1$ should be be optimal.

Our analysis relies upon taking the Fourier transform in the time variable and
using recent estimates for the kernel of the CFIE done by Chandler-Wilde and
Monk, and Baskin, Spence, and Wunsch, see~\cite{BSW,ChandlerMonk}. The results
of~\cite{BSW} and our own also rely on several modern improvements to the
classical results in scattering theory alluded to above.

\section{Time-harmonic Analysis}

Recall that 
\begin{equation}
  g_k(x)=\frac{e^{ik|x|}}{4\pi |x|},
\end{equation}
is the outgoing fundamental solution for $(\Delta+k^2),$ with the corresponding
single and
double layer potentials defined by:
\begin{equation}
  S_kf(x)=\int\limits_{\pa D}g_k(x-y)f(y)dS_y;\quad
D_kf(x)=\int\limits_{\pa D}\pa_{\bn_y}g_k(x-y)f(y)dS_y.
\end{equation}
With respect to the complex linear pairing, $\langle u,v\rangle=\int_{\pa
  D}uvdS,$ the single layer $S_k$ is self-dual and
\begin{equation}
  D'_kf(x)=\int\limits_{\pa D}\pa_{\bn_x}g_k(x-y)f(y)dS_y.
\end{equation}
If we take the Fourier
transform in time:
\begin{equation}
  \hmu(x,k)=\int\limits_{0}^{\infty}\mu(x,t)e^{-itk}dt,
\end{equation}
then equation~\eqref{eqn3} becomes
\begin{equation}
   \left[\frac{\hmu(x,k)}{2}+D_k\hmu(x,k)-S_k[i a(y)k\hmu-b(y)\hmu](x,k)\right]=\hg(x,k)
\end{equation}
As before the limit is taken from the exterior of $D.$ Here and in the sequel
we assume that $g(x,t)=0$ for $t\leq 0;$ by causality $\mu(x,t)$ can also be
taken to vanish for $t\leq 0.$

The combined field operators, acting on functions on $\Gamma,$  given by
\begin{equation}
  A(a,b;k)=\frac{I}{2}+D_k-iS_k[ka(y)+ib(y)],
\end{equation}
are well known to be Fredholm operators of second kind, provided that $\Gamma$
is at least $\cC^1.$ For functions $a(y)>0$ and $b(y)> 0$ for all $y\in\Gamma,$ it
can be shown, by a simple integration by parts argument, that $A(a,b;k)$ is
invertible for $k$ in the closed upper half plane. To prove this we use that
the dual of $A(a,b;k)$ with respect to the complex linear pairing is
\begin{equation}
  A'(a,b;k)=\frac{I}{2}+D'_k-i(a(y)k+ib(y))S_k.
\end{equation}

Suppose that for some $\Im k\geq 0,$ there is a non-trivial solution to
$A(a,b;k)f=0.$ By the Fredholm theory this means that there is also a
non-trivial solution $h$ to $A'(a,b;k)h=0.$ Let
\begin{equation}
  u(x)=S_k h(x).
\end{equation}
We will have occasion to consider the 
function $u(x)$ in both the interior of $D$ and its exterior. To make 
this distinction clear, we will refer to the corresponding restriction of
$u$ by $u_-$ and $u_+$, respectively.

Observe now that $u_-\neq 0.$ If it were, then $u_+\restrictedto_{\pa D}=0$
as well. In this case $u_+$ would be an solution to
$(\Delta+k^2)u_+=0$ with vanishing Dirichlet data, satisfying the Sommerfeld
radiation condition and it would therefore also
be zero, which is impossible. Thus, when $\Im k\geq 0,$ this gives a non-trivial solution in $D$ to
\begin{equation}\label{eqn13}
  (\Delta+k^2)u_-=0\text{ with }\pa_{\bn}u_-(x)-i(ka(x)+ib(x))u_-(x)=0.
\end{equation}
The boundary condition is equivalent to the fact that $A'(a,b;k)h=0.$

Using the boundary condition, and integration by parts we obtain:
\begin{equation}\label{eqn14.1}
\begin{split}
  0&=\int\limits_{D}(\Delta
  u_-+k^2u_-)\bar{u}_-dx=\int\limits_{D}[k^2|u_-|^2-|\nabla
  u_-|^2]dx+\int\limits_{\pa D}\pa_{\bn}u_-\bar{u}_-\\
&=\int\limits_{D}[k^2|u_-|^2-|\nabla
  u_-|^2]dx+i\int\limits_{\pa D}(ka(x)+ib(x))|u_-|^2dS.
\end{split}
\end{equation}
Let $k=k_1+ik_2$, and take the real and imaginary parts to see that this
implies
\begin{equation}\label{eqn15.1}
\begin{split}
  &k_1\left[2k_2\int\limits_{D}|u_-|^2dx+\int\limits_{\pa D}a(x)|u_-|^2dS\right]=0\\
&\int\limits_{D}[(k_1^2-k_2^2)|u_-|^2-|\nabla
  u_-|^2]dx-\int\limits_{\pa D}(k_2a(x)+b(x))|u_-|^2dS=0.
\end{split}
\end{equation}
The first relation shows that if $k_1\neq 0$ and $k_2>0,$ then $u_-=0$ as
well. If $k_1=0$ and $k_2\geq 0,$ then the second equation shows that $u_-=0.$
If $k_2=0,$ but $k_1\neq 0,$ then the first relation implies that
$u_-\restrictedto_{bD}=0;$ the boundary condition in~\eqref{eqn13} then implies
that $\pa_{\bn}u_-=0$ as well.  It then follows from Green's formula that
$u_-=0.$ This proves the following basic theorem:
\begin{theorem}
  If $a(x)$ and $b(x)$ are positive functions defined on $\pa D,$ a compact
  $\cC^1$-surface, then the operators $A(a,b;k)$ and $A'(a,b;k)$ are invertible
  for $k$ with $\Im k\geq 0.$
\end{theorem}
\begin{remark} If $b=0,$ then $A(a,b;0)$ has a non-trivial nullspace. Such a
  nullspace would generically destroy any possibility for exponential decay in
  the source function $\mu(x,t).$ In Kress' early paper \cite{Kress3} on the combined field
  operator, he shows that the optimal result (for the disk) when $k$ is close to
  zero results from taking $b=-\frac i2+O(k^2)$ in our notation.
  This is somewhat at odds with our choice to take $b$ real
  and positive. 

  Careful examination shows that Kress' choice only works if $k_1\geq 0:$ In
  the set $k_2<|k_1|$ the real part of the quadratic form in~\eqref{eqn14.1} is
  indefinite (regardless of the values that $a$ and $b$ take) and so to obtain
  the desired result we need to employ the imaginary part of the quadratic
  form.  If $b=b_1+ib_2,$ then this form would become
  \begin{equation}
    k_1\left[2k_2\int_D|u_-|^2dx+a\int_{\pa D}|u_-|^2dS\right]-b_2\int_{\pa D}|u_-|^2dS.
  \end{equation}
  In order for the term in the brackets to be definite where $k_2>0,$ we need
  to take $a>0.$ In order for this expression to be definite in both components
  of the set $0< k_2<|k_1|$ it is clearly necessary to take $b_2=0.$ Of course
  one could use functions of $k$ more complicated than $ak+ib$ as the
  coefficient of $S_k,$ but this would considerably complicate the relationship
  between the solutions in the frequency  and the time domains.
\end{remark}

\section{Scattering Theory}
Let $(a(x), b(x))$ be positive functions defined on $\pa D.$ In the recent
paper of Baskin, Spence, and Wunsch, see Theorem 1.10 of~\cite{BSW}, it is shown that there
is a positive number $\beta_1$ so that the operator $(\Delta+k^2)u$ acting on
data in $D$ that satisfies the boundary condition
$(\pa_{\bn}u-i(a(x)+ib(x))u)=0,$ is invertible for $\Im k>-\beta_1.$ This
result does \emph{not} assume that $\pa D$ is non-trapping. Using this result,
along with modern refinements of theorems of Lax and Phillips and Lax, Morawetz and
Phillips \cite{LMP63,LP}, we can prove the following result:
\begin{theorem} For $n\in\bbN,$
  let $D$ be a  non-trapping region in $\bbR^{2n+1}.$ If
  $a$ and $b$ are both positive, then there is a positive number $\alpha$ so that
  $A(a,b;k)$ is invertible in $\Im k>-\alpha.$
\end{theorem}
\begin{proof}
  The facts we use, in addition to the result of~\cite{BSW} are
  \begin{enumerate}
  \item If $D$ is a non-trapping region then the generator, $B,$ of the
    compressed wave-semigroup $Z(t)$ of Lax and Phillips has its spectrum,
    $\sigma(B),$ in a half plane of the form $\Im k\leq -\beta_0<0.$
  \item A number $k$ belongs to $\sigma(B)$ if and only if there is an
    ``eventually outgoing'' solution $v_+$ to $(\Delta+k^2)v_+=0$ with
    $v_+\restrictedto_{\pa D}=0.$ See Theorem V-4.1 in~\cite{LP}.
  \end{enumerate}
  By definition (see \cite{LP} pp. 126-7) a solution is eventually outgoing (even for
  $\Im k<0$) if and only if it can be represented in the form
\begin{equation}\label{eqn8}
  v_+(x)=\int\limits_{\pa D}[\pa_{\nu_y}v_+(y)g_k(x-y)-v_+(y)\pa_{\nu_y}g_k(x-y)]dy,
\end{equation}
with $\nu_y=-\bn_y,$ the outer normal to $D^c.$ See Theorem IV-4.3
in~\cite{LP}. Here $g_k(x)$ is the ``outgoing'' fundamental solution for
$\Delta+k^2$ in $\bbR^{2n+1}.$

We argue as before: if for some $k,$ with $\Im k<0,$ the operator $A(a,b;k)$
has a non-trivial null-space, then so does $A'(a,b;k).$ Let $\varphi\neq 0$ be
in this latter nullspace and set
\begin{equation}
  u=\int\limits_{\pa D}\varphi g_k(x-y)dS_y.
\end{equation}
Unlike the case where $\Im k\geq 0,$ we do not know, a priori, that $u_-\neq
0.$ Indeed, there are just the two cases to consider: $u_-=0$ or not. If $u_-\neq 0,$ then 
$u_-$ is a non-trivial solution to the boundary value problem considered
in~\cite{BSW}, and therefore $\Im k\leq -\beta_1.$ 

The other possibility is that $u_-=0.$ This implies that $u_+\restrictedto_{\pa
  D}=0$ as well, since $u$ is continuous across $\pa D.$ However $u_+$ need not
be zero, since it is exponentially growing at infinity. Indeed, if $\varphi\neq
0,$ and $u_-=0,$ then the jump conditions for $\pa_{\bn}u_{\pm}$ show that
$u_+$ cannot be zero. Because $u_+\restrictedto_{\pa D}=0,$ and
$\pa_{\bn}u_+=-\varphi,$ it follows that $u_+$ satisfies~\eqref{eqn8} and therefore
that $u_+$ is an eventually outgoing solution in the sense of Lax and Phillips.
Hence $k$ must belong to the spectrum of $B,$ and therefore $\Im k\leq
-\beta_0.$ This shows that if $A(a,b;k)$ is non-invertible, then $\Im k\leq
\max\{-\beta_1,-\beta_0\}=-\alpha$.
\end{proof}

The converse is also true.
\begin{theorem}
 If $k$ is a scattering pole for the Dirichlet Laplacian, then
$A(a,b;k)$ has a non-trivial null-space.
\end{theorem}
\begin{proof}
 The theorem of Lax and Phillips states
that $k$ is a scattering pole if and only there is an outgoing solution to the
BVP
\begin{equation}
  (\Delta+k^2)u=0\text{ in }D^c\text{ and }u\restrictedto_{bD}=0;
\end{equation}
the outgoing condition in this case is equivalent to
\begin{equation}
  u(x)=-\int\limits_{bD}g_k(x-y)\pa_{\bn}u(y)dS_y\text{ for }x\in D^c.
\end{equation}
If we let $\varphi=\pa_{\bn}u,$ then $S_k\varphi$ vanishes on $bD$ and
\begin{equation}
  \frac{\varphi}{2}-D_k'\varphi=\varphi,
\end{equation}
from which it is immediate that
\begin{equation}
  A'(a,b;k)\varphi=0,
\end{equation}
which proves the claim. Thus every scattering resonance occurs among the set
$\{k:\:\Ker A(a,b;k)\neq 0\},$ for any  choice of positive $a$ and $b.$ 
\end{proof}

\begin{remark} These arguments apply generally to identify the frequencies for
  which the null-space of
  any combination of single and double layer potentials is non-trivial. The
  only cases not explicitly covered are those of the single and double layers
  alone. One easily establishes that $S_k$ fails to be invertible for $k^2$ in
  the Dirichlet spectrum of $D$ union with the scattering poles of the
  Dirichlet operator in $\Omega.$ The (exterior) double layer $\Id/2+D_k$ fails
  to be invertible for $k^2$ in the Neumann spectrum of $D$ union with the scattering poles of the
  Dirichlet operator in $\Omega.$ 
\end{remark}

Spence et al. show that that the resolvent kernel for the interior impedance
problem satisfies the norm estimate $\|R_{a,b}(k^2)\|\leq \frac{C}{1+|k|}$ for
$k$ in a strip around the real axis. It is not clear what the analogous bound
is for the analytic continuation of the exterior Dirichlet problem
$\|R_{\Dir}(k^2)\|$ into the lower half plane. This operator maps compactly
supported data into functions that have exponential growth at infinity. For
later applications we will need to determine bounds on $\|A(a,b;k)^{-1}\|$ for
$k$ in the lower half plane. In particular it will be important to establish a
bound of the form $\|A(a,b;k)^{-1}\|\leq M(1+|k|)^m.$

Suppose that $g(x,t)$ is data for the wave equation on the boundary of $D$ with
support for $t\in [0,T],$ and let
\begin{equation}
  \hg(x,k)=\int\limits_{0}^{\infty}g(x,k)e^{tk}dt.
\end{equation}
The source term on the boundary as a function of $t$ will be given by the
inverse Fourier transform
\begin{equation}
  \mu(x,t)=\frac{1}{2\pi}\int\limits_{-\infty}^{\infty}e^{t(ik)}[A(a,b;k)]^{-1}\hg(x,ik)dk.
\end{equation}
If this is the case, then the corresponding time-domain
representation of a solution to the wave-equation is 
\begin{equation}
 u(x,t)= \cD_t\mu+\cS_t(a\pa_t+b)\mu,
\end{equation}
with $\cS_t$ and $\cD_t$ the retarded potential single and double layers for
the wave equation. 

Assume  that for any $\epsilon>0$ and $\Im
k>-(\alpha-\epsilon),$ we have an estimate like 
$$\|[A(a,b;k)]^{-1}\|\leq
C_{\epsilon}(1+|k|)^m.$$ 
For $g$ with sufficient smoothness in $t$ we can  deform the contour
to conclude that, for $\sigma<\alpha,$ we have the representation:
\begin{equation}
  \mu(x,t)=\frac{1}{2\pi}\int\limits_{-\infty}^{\infty}e^{t(ik-\sigma)}
[A(a,b;k-i\sigma)]^{-1}\hg(x,ik-\sigma)dk,
\end{equation}
and therefore
\begin{equation}
  \|\mu(\cdot,t)\|\leq Ce^{-\sigma t}\|g\|_{m,T}.
\end{equation}
For conveniently defined norms, this is essentially what we wanted to prove. 

Of course this may not give the ``optimal'' rate of decay. It should be noted
that the rate of decay $\alpha$ is directly related to two spectral invariants,
one being the genuine scattering resonances, and the other the spectrum of
the BVP for $\Delta$ on $L^2(D)$ with impedance BC
\begin{equation}
  \pa_{\bn}u-i(ak+ib)u=0\text{ on }\pa D.
\end{equation}
At least near the negative imaginary axis, we can  push this down into
the lower half plane so that the first eigenvalue encountered is a genuine
scattering resonance.  This follows from a careful examination of the 
integration by parts argument given above. 

Let $a$ and $b$ be positive constants. We assume that $(\Delta u+k^2 u)=0,$ and
\begin{equation}
  \pa_{\bn}u-i(ak+ib)u\restrictedto_{\pa D}=0.
\end{equation}
If we integrate by parts we see that, with $k=k_1+ik_2:$
\begin{equation}\label{eqn18}
\begin{split}
 0 &=\int\limits_{D}((\Delta u+k^2 u)\bar{u}dx\\
&=-\int\limits_{D}[|\nabla
u|^2+(k_2^2-k_1^2)|u|^2]dx-(ak_2+b)\int\limits_{\pa D}|u|^2dS+\\
&ik_1\left(2k_2\int\limits_{D}|u|^2dx+a\int\limits_{\pa D}|u|^2dS\right).
\end{split}
\end{equation}
If $k_1\neq 0,$ the we can substitute from the imaginary part the fact that
\begin{equation}
 \int\limits_{\pa D}|u|^2dS=- \frac{2k_2}{a}\int\limits_{D}|u|^2dx
\end{equation}
into the real part to see that
\begin{equation}
 0= \int\limits_{D}|\nabla u|^2dx-\left[\frac{a(k_1^2+k_2^2)+2bk_2}{a}\right]\int\limits_{D}|u|^2dx.
\end{equation}
Evidently the quantity in the brackets must be positive, which implies that
\begin{equation}
k_1^2+\left(k_2+\frac{b}{a}\right)^2>\left(\frac{b}{a}\right)^2.  
\end{equation}
Thus the impedance boundary value problem has a pole free region lying in the
disk minus the imaginary axis
\begin{equation}
k_1^2+\left(k_2+\frac{b}{a}\right)^2\leq
\left(\frac{b}{a}\right)^2\setminus\{k_1 = 0\}.  
\end{equation}
Along the imaginary axis its clear that $k_2<-b/a$ is needed for~\eqref{eqn18}
to hold. A somewhat better result can be obtained by using the fact
that there is an estimate of the form
\begin{equation}
  \int\limits_{\pa D}|u|^2dS\leq C_1\left[\int\limits_{D}[|\nabla
u|^2+|u|^2]dx\right],
\end{equation}
which implies that, if $k_1=0,$ then
\begin{equation}
  k_2\leq\max\{\tau:1+(a\tau+b)C_1\leq 0 \text{ and }\tau^2+C_1(a\tau+b)\leq 0\}.
\end{equation}
Taking $b/a$ large, we can arrange to have the pole-free region encompass as much of the lower half plane as we want.

\section{The Round Sphere}
In this section we consider the case that $\Gamma$ is equal to the unit sphere
in $\bbR^3.$ We take $a$ to be a constant $\alpha$ and $b$ a constant $\beta.$
Since everything commutes with the action of the rotation group, we can analyze
this problem one spherical harmonic subspace at a time.

\subsection{Integral Equations on Spherical Harmonic Subspaces}

We need to analyze the action of the operator
\begin{multline}\label{eqn36}
  G^{\alpha,\beta}\mu(x,t)=\frac{\mu(x,t)}{2}+\frac{1}{4\pi}\int\limits_{\Gamma}\frac{n_y\cdot(x-y)}{|x-y|}\left(
\frac{\mu(y,\tau)}{|x-y|^2}+\frac{\pa_t\mu(y,\tau)}{|x-y|}\right)dS_y+\\
\frac{1}{4\pi}\int\limits_{\Gamma}\frac{\alpha\pa_t\mu(y,\tau)+\beta\mu(y,\tau)}{|x-y|}dS_y,
\end{multline}
where $\tau=t-|x-y|,$ on data of the form $\mu(y,t)=Y^m_n(y)f(t).$ Notationally
it's easier to study the general class of operators of the form:
\begin{equation}
  K(Y^m_nf)(x,t)=\frac{1}{4\pi}\int\limits_{\Gamma}Y^m_n(y)f(t-|x-y|)k(x\cdot y)dS_y.
\end{equation}
Up to taking time derivatives of $f,$ all our operators take this form. 

The key observation is that if we let $x^{\bot}$ be a unit vector orthogonal to
$x,$ then we can define a coordinate system $(\theta,\phi)$ on the unit sphere
by setting
\begin{equation}
  y=\cos\theta x+\sin\theta R_{\phi}x^{\bot}.
\end{equation}
Here, $R_{\phi}$ is a rotation through angle $\phi$ about $x.$ Substituting, we
see that
\begin{multline}
  K(Y^m_nf)(x,t)=\frac{1}{4\pi}\int\limits_{0}^{\pi}\int\limits_{0}^{2\pi}Y^m_n(\cos\theta
  x+\sin\theta R_{\phi}x^{\bot})\\
f(t-\sqrt{2(1-\cos\theta)})k(\cos\theta)\sin\theta d\phi d\theta.
\end{multline}
A simple calculation shows that
\begin{equation}
  \frac{1}{2\pi}\int\limits_{0}^{2\pi}Y^m_n(\cos\theta
  x+\sin\theta R_{\phi}x^{\bot})d\phi= Y^m_n(x)P_n(\cos\theta).
\end{equation}
Here $P_n(z)$ is the degree $n$ Legendre polynomials normalized so that $P_n(\pm 1)=(\pm 1)^n.$ Therefore we have that
\begin{equation}
  K(Y^m_nf)(x,t)=\frac{Y^m_n(x)}{2}\int\limits_{0}^{\pi}P_n(\cos\theta)f(t-\sqrt{2(1-\cos\theta)})
k(\cos\theta)\sin\theta d\theta.
\end{equation}
Let $s=\sqrt{2(1-\cos\theta)},$ and $\tk(s)=\sqrt{2(1-s)}k(s),$ to obtain:
\begin{equation}
   K(Y^m_nf)(x,t)=\frac{Y^m_n(x)}{2}\int\limits_{0}^{2}P_n(1-s^2/2)f(t-s)
\tk(1-s^2/2) ds.
\end{equation}

To represent $G^{\alpha,\beta}$ acting on data of this type we need to determine what $\tk$ is
for each of the terms in~\eqref{eqn36}.  The kernels are
\begin{equation}
  \begin{split}
    \frac{n_y\cdot (x-y)}{|x-y|^3}=\frac{x\cdot y-1}{[2(1-x\cdot y)]^{\frac
        32}}&\leftrightarrow \tk(s)=\frac{-1}{2}\\
 \frac{n_y\cdot (x-y)}{|x-y|^2}=\frac{x\cdot y-1}{[2(1-x\cdot
   y)]}&\leftrightarrow \tk(s)=-\frac{\sqrt{2(1-s)}}{2}\\
\frac{1}{|x-y|}=\frac{1}{\sqrt{2(1-x\cdot
   y)}}&\leftrightarrow \tk(s)=1
  \end{split}
\end{equation}
Acting on $Y^m_nf$ we get
\begin{multline}
  G^{\alpha,\beta}(Y^m_nf)(x,t)=\\
\frac{Y^m_n(x)}{2}\Bigg[f+\frac{1}{2}
\int\limits_0^2P_n(1-s^2/2)\large\{f(t-s)\left(-1\right)+\\
\pa_sf(t-s)\left(s\right)-2\alpha\pa_sf(t-s)+2\beta
f(t-s)\large\}ds\Bigg].
\end{multline}
Note that we have replaced $\pa_t$ with $-\pa_s.$ Integrating by parts with respect to $s$ gives
\begin{multline}
  G^{\alpha,\beta}(Y^m_nf)(x,t)=
Y^m_n(x)\Bigg[\frac{(1+\alpha)}{2}f(t)+(-1)^n\frac{(1-\alpha)}{2}f(t-2)-\\\frac{1}{4}\int\limits_0^2
\left[(2-2\beta)P_n(1-s^2/2)-s(s-2\alpha)P'_n(1-s^2/2)]\right]f(t-s)ds\Bigg].
\end{multline}
Denote the 1-d  operator in the brackets by $G_{n}^{\alpha,\beta}f.$
If we let $\alpha=\beta=1,$ we get a very simple integral equation of the second
kind to solve for $f:$
\begin{equation}
  G_n^{1,1}f_n^m(t)=f_n^m(t)+\frac{1}{4}\int\limits_0^2
s(s-2)P'_n(1-s^2/2))f_n^m(t-s)ds=g_n^m(t).
\end{equation}
If $n=0,$ then this is simply $f_0^0(t)=g_0^0(t).$

\subsection{The Roles of $\alpha$ and $\beta$}
The equation 
\begin{equation}
  G_0^{\alpha,1}f(t)=\frac{(1+\alpha)}{2}f(t)+\frac{(1-\alpha)}{2}f(t-2)=g(t),
\end{equation}
is quite informative as regards the rate of decay and regularity of the
solution. Assuming that $0<\alpha<1,$ we simplify this equation to obtain
\begin{equation}
  f(t)+\frac{1-\alpha}{1+\alpha}f(t-2)=\frac{2}{1+\alpha}g(t).
\end{equation}
If $g(t)=0$ for $t<0,$ then we can let $f(t)=0,$ for $t<0$ as well. The solution to this equation is formally given
by the infinite series
\begin{equation}
  f(t)=\frac{2}{1+\alpha}\sum_{j=0}^{\infty}(-\lambda)^jg(t-2j),\text{ with }\lambda=\frac{1-\alpha}{1+\alpha}.
\end{equation}
For data supported in $[0,\infty)$  this sum is finite for any $t.$ In particular:
\begin{equation}
  f(t)=\frac{2}{1+\alpha}g(t)\text{ for }t\in (-\infty,2).
\end{equation}

If $g$ is supported in $[0,2M]$ then the sum can be made more explicit:
\begin{equation}
  f(t)=\frac{2}{1+\alpha}\sum_{j=\left[\frac{t}{2}-M\right]\wedge 0}^{\left[\frac{t}{2}\right]}(-\lambda)^jg(t-2j).
\end{equation}
Eventually $f$ satisfies an estimate of the form
\begin{equation}
  |f(t)|\leq C_{M,\lambda}\left(\frac{1-\alpha}{1+\alpha}\right)^{\frac{t}{2}}\|g\|_{L^{\infty}}.
\end{equation}
This shows how the choice of $\alpha$ affects the decay of the source term on
the boundary. Furthermore, if $g$ is smooth but its support does not lie in an
interval of the form $[2l, 2(l+1)],$ then the solution $f$ typically has a jump
discontinuity at every positive even integer.

The $n=0, \beta=0$ case can be solved by using Laplace transform:
\begin{equation}
  \cL g(\tau)=\gamma^{\alpha}_0(\tau)\cL f(\tau),
\end{equation}
where
\begin{equation}
  \gamma^{\alpha}_0(\tau)=
\frac{1}{2}\left[1+e^{-2s}+\left(\alpha-\frac{1}{s}\right)(1-e^{-2s})\right].
\end{equation}
Note, however, that no matter what value $\alpha$ takes, this multiplier has a
root at $s=0.$ This explains why we need to take $\beta>0$ in order to get
exponential decay.

\subsection{Numerical Illustrations}

To provide concrete examples of the implications of our analysis, we have carried out
representative numerical simulations in the case $n=0$. We consider two choices
for the Dirichlet data:
\begin{eqnarray*}
g(t)=8 \sin{(50t)} \cdot e^{-40(t-1)^2} & & {\rm Oscillatory}, \\
g(t)=8 e^{-40(t-1)^2} & & {\rm Non-oscillatory} .
\end{eqnarray*}
We also consider three choices for the parameters in the integral equation:
\begin{eqnarray*}
a=\alpha=0, & & b=\beta=0 . \\
a=\alpha=1, & & b=\beta=0 . \\
a=\alpha=1, & & b=\beta=\frac {1}{2} .
\end{eqnarray*}
(Note that the optimal choice for the sphere, $a=b=1$, leads to a trivial equation when $n=0$ so
we avoid it.) 

Our numerical method is based on the standard Adams predictor-corrector idea. Recall that
for this special case the equation to be solved is
\begin{equation}
\mu(t) + \left( \frac{1-\alpha}{1+\alpha} \right) \mu(t-2) -
\left( \frac{1-\beta}{1+\alpha} \right) \int_{t-2}^t \mu(s)ds = \frac {2}{1+\alpha} g(t) . \label{mode0eq}
\end{equation}
Introducing a time step $\Delta t$ we rewrite in correction form:
\begin{eqnarray}
\mu(t+\Delta t) & = & \mu(t) - \left( \frac{1-\alpha}{1+\alpha} \right) 
\left(\mu(t+\Delta t-2)-\mu(t-2) \right) \nonumber \\ & &
+ \left( \frac{1-\beta}{1+\alpha} \right) \int_{t}^{t+\Delta t} \mu (s)ds
- \left( \frac{1-\beta}{1+\alpha} \right) \int_{t-2}^{t+\Delta t-2} \mu (s)ds \nonumber \\ 
& & + \frac {2}{1+\alpha} \left( g(t+\Delta t) -g(t) \right) . \label{CorrEq} 
\end{eqnarray} 

In the examples which follow we interpolate past solution data on an interval
$(j \Delta t, (j+1)\Delta t)$ by the Lagrange polynomial of degree $5$ using
the approximate solution at $t=k \Delta t$, $k=j-3, \ldots , j+2$. Precisely the
interpolant is used to calculate $\mu(t-2)$, $\mu(t+\Delta t-2)$, $\int_{t-2}^{t+\Delta t-2} 
\mu(s) ds$ appearing on the right hand side of (\ref{CorrEq}). The future
integral uses the $6$th order Adams-Bashforth-Moulton predictor-corrector
method. That is, we compute a predicted value $\mu^{(p)}(t+\Delta t)$ using the degree $5$
Lagrange interpolant of $\mu(t-k \Delta t)$, $k=0, \ldots ,5$ to approximate $\int_{t}^{t+\Delta t} \mu(s)ds$.
We then compute $\mu(t+\Delta t)$ using the interpolant of $\mu^{(p)}(t+ \Delta t)$ and $\mu(t-k\Delta t)$,
$k=0, \ldots ,4$ inside the integral. We have verified that convergence
at $6$th order is generally obtained (in some experiments the observed rate was reduced to $5$ when
$\alpha=\beta=0$) and we have also cross-checked the results with those
computed using $2$nd and $4$th order solvers constructed the same way. The solutions displayed in the
graphs were calculated with a time step $\Delta t=97/6400$ for the non-oscillatory data and
$\Delta t=97/12800$ for the oscillatory data. Comparisons with a coarsened computation suggest
that the solutions are accurate to at least $7$ digits in the non-oscillatory case and $4$-$6$
digits when the data was oscillatory. Here the maximum errors in the oscillatory case
are two orders of magnitude larger
when $\alpha=\beta=0$; looking at the solutions below we attribute the difference to dispersion
effects. 

Solutions of the integral equation
for the choice $\alpha=\beta=0$ are displayed in Figure \ref{FigOa0b0}.
We see that in
the case where the data is oscillatory $\mu$ does not decay but oscillates at high
frequency. When the data is non-oscillatory we observe linear growth. In each case we note that
the solution of the wave equation
itself decays exponentially in time, so, as our analysis shows, the long time behavior of the density
$\mu(t)$ is quite different than that of the function it represents. 

\begin{figure}[htb]
\begin{center}
\includegraphics[width=0.45\textwidth]{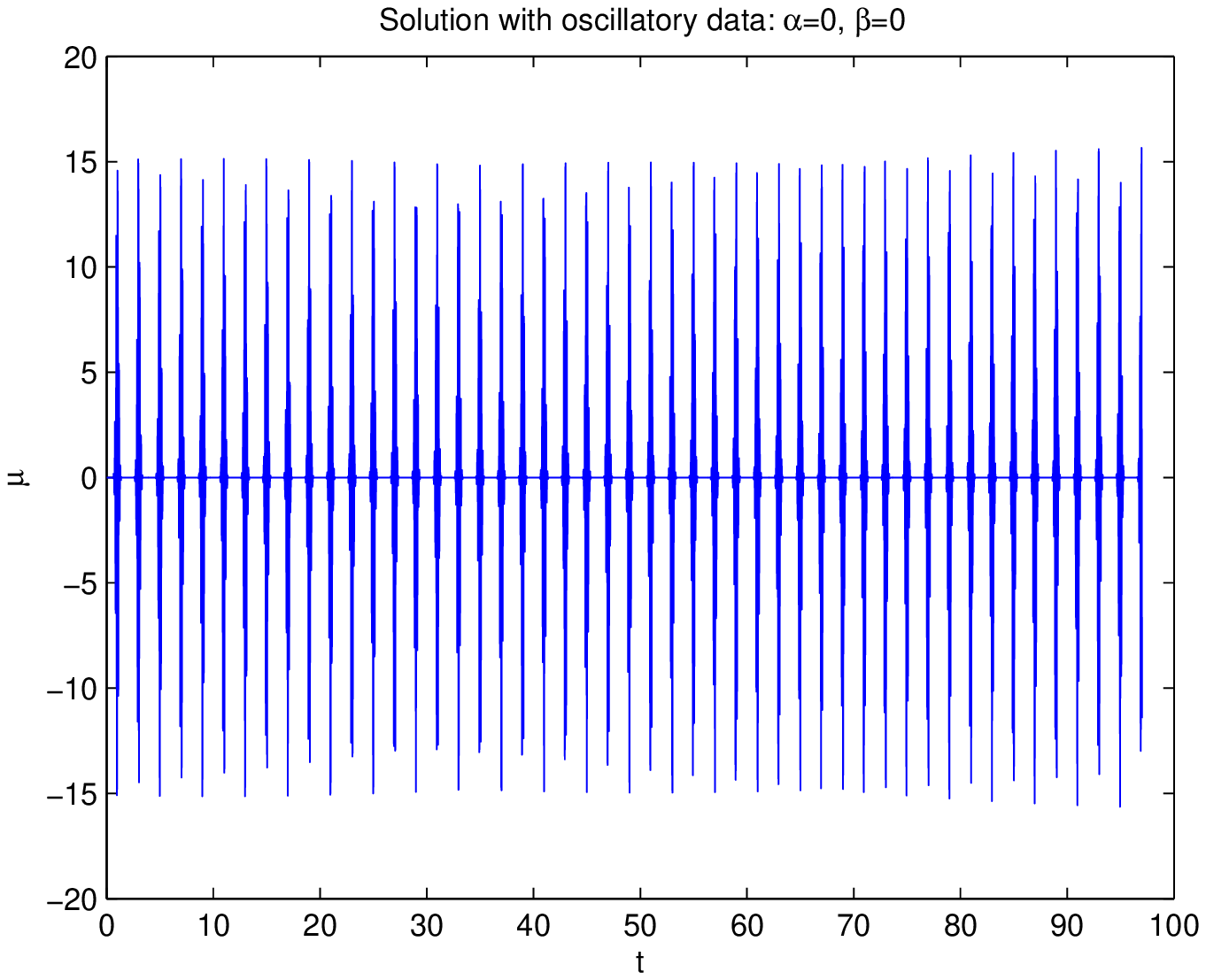}
\includegraphics[width=0.45\textwidth]{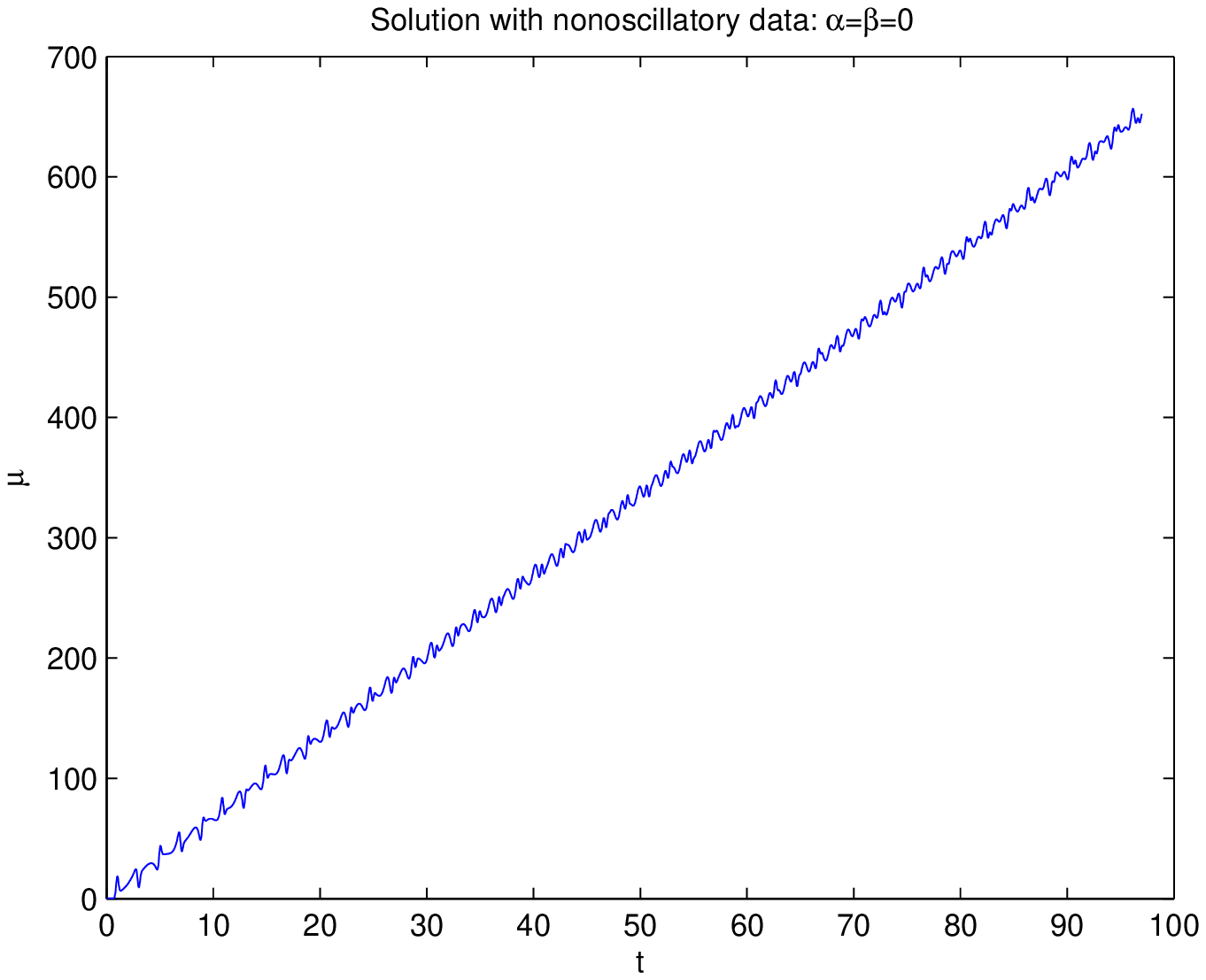}
\caption{Density $\mu(t)$ at mode $0$ with oscillatory data 
(left) and non-oscillatory data (right) for $\alpha=\beta=0$.\label{FigOa0b0}}
\end{center}
\end{figure}

Solutions of the integral equation
for the choice $\alpha=1$, $\beta=0$ are displayed in Figure \ref{FigOa1b0}. We see that in
the case where the data is oscillatory $\mu$ apparently decays in time. 
However a closer look indicates that it in fact approaches a very small steady state. 
When the data is non-oscillatory $\mu$ clearly approaches an $O(1)$ steady state,
again in contrast with the solution of the wave equation itself. This behavior corresponds to the
pole at the origin discussed above. 

\begin{figure}[htb]
\begin{center}
\includegraphics[width=0.45\textwidth]{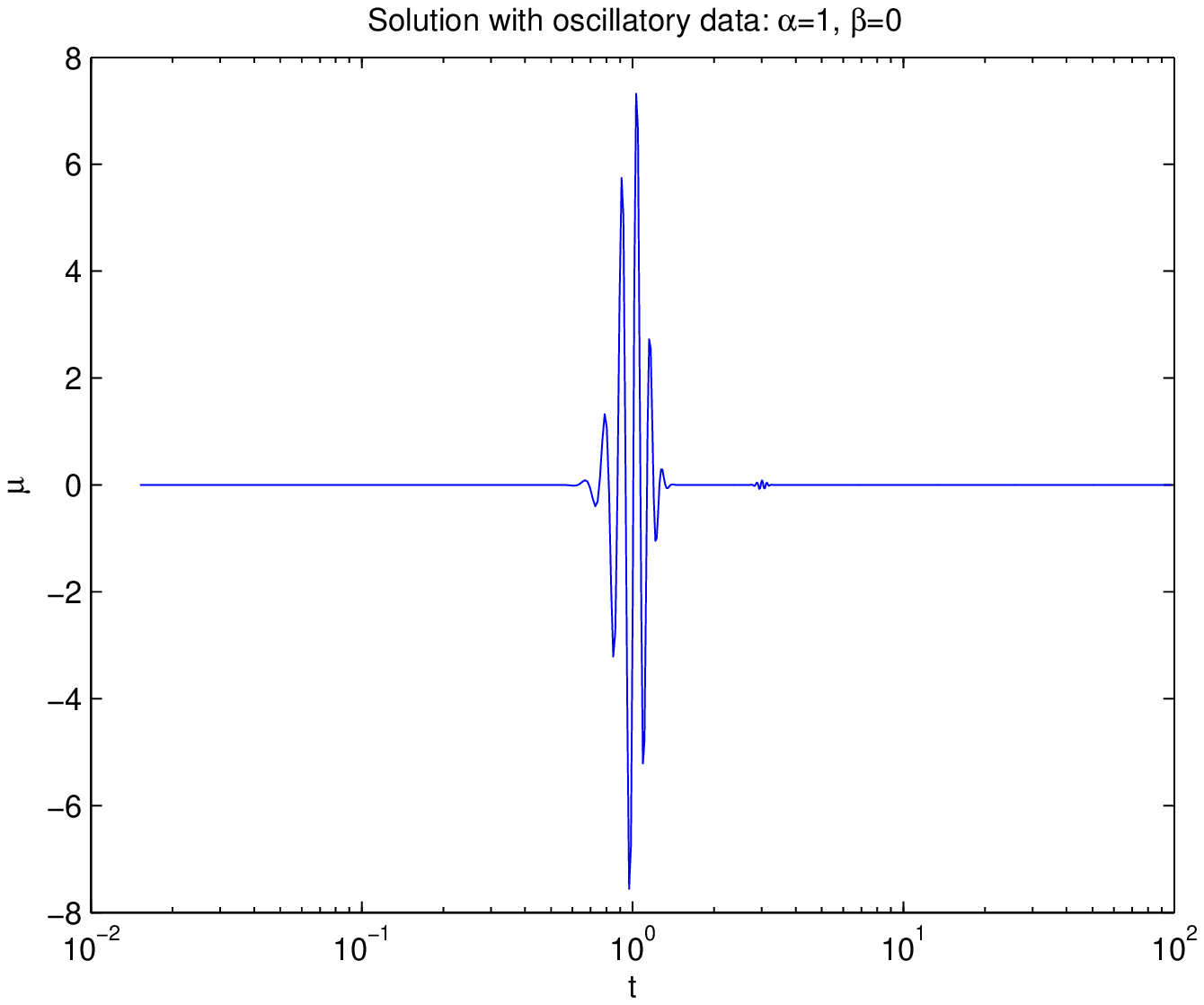}
\includegraphics[width=0.45\textwidth]{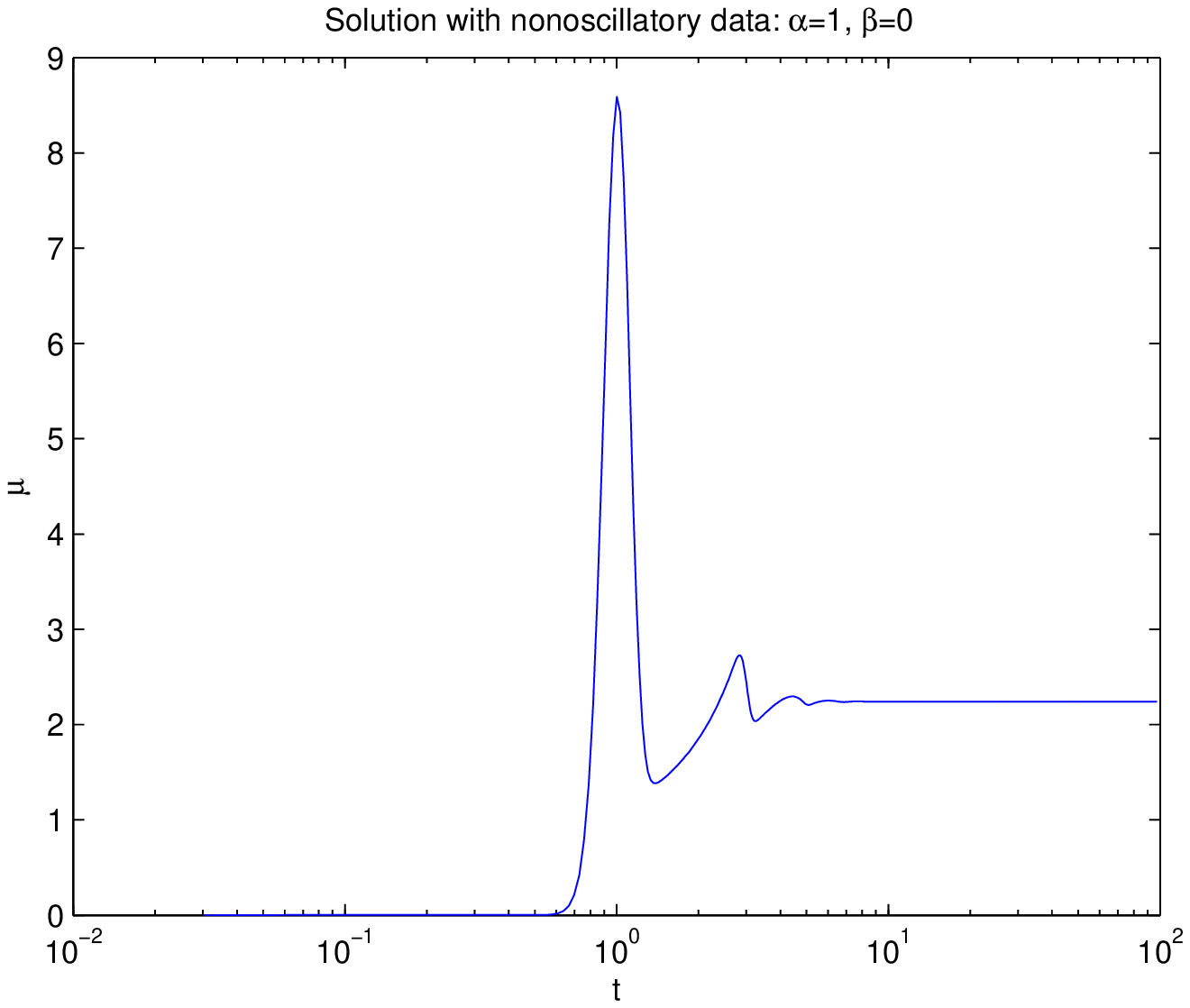}
\caption{Density $\mu(t)$ at mode $0$ with oscillatory data (left) and
non-oscillatory (right) for $\alpha=1$, $\beta=0$.\label{FigOa1b0}}
\end{center}
\end{figure}

Lastly we consider the solutions which arise when $\alpha=1$, $\beta=\frac {1}{2}$;
these are displayed in Figure \ref{FigOa1bh}. Here in each case we have exponential
decay of $\mu$ in time.

\begin{figure}[htb]
\begin{center}
\includegraphics[width=0.45\textwidth]{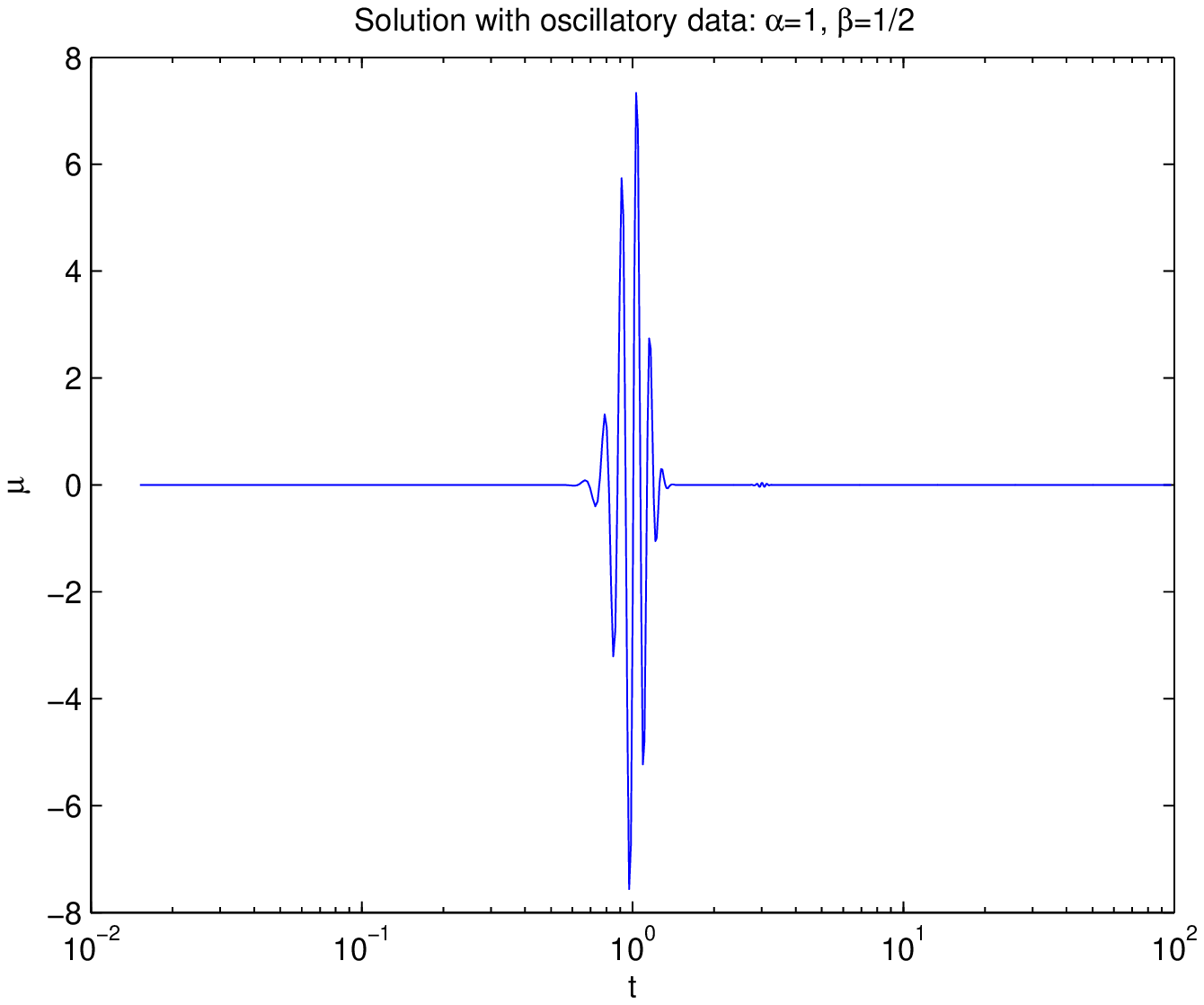}
\includegraphics[width=0.45\textwidth]{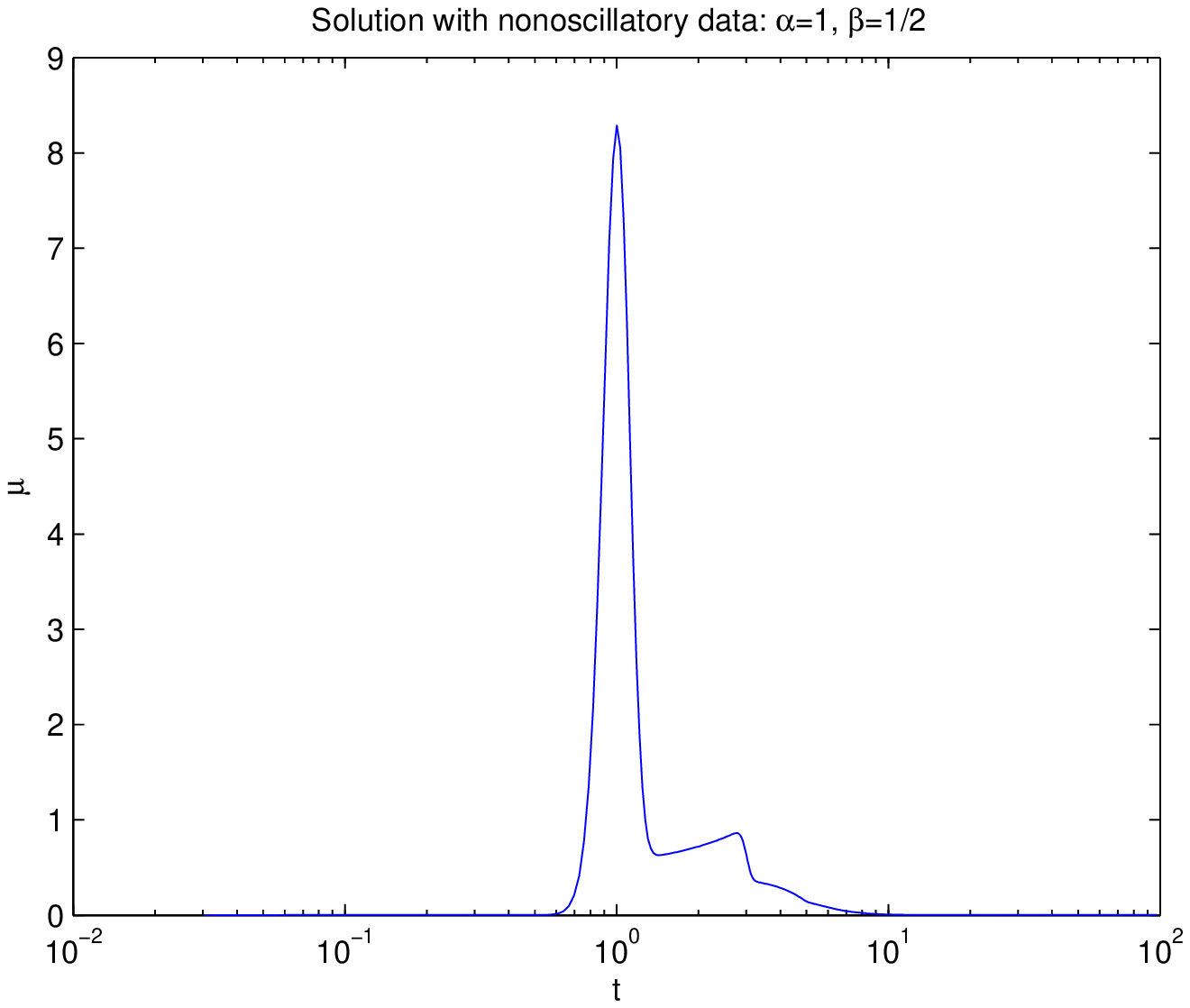}
\caption{Density $\mu(t)$ at mode $0$ with oscillatory data (left) 
and non-oscillatory data (right) for $\alpha=1$, $\beta=\frac {1}{2}$.\label{FigOa1bh}}
\end{center}
\end{figure}

\section{High Frequency Asymptotics for Convex Bodies}\label{sec5}
In the previous section we saw that, at least for round spheres, a principal
determinant of the decay rate for the solution of the integral equations on the
boundary, $G^{\alpha,\beta}_nf=g,$ is the delay term $\frac
12[(1+\alpha)f(t)+(-1)^n(1-\alpha)f(t-2)].$ Choosing $\alpha=1$ removes the
sharp contribution from the antipodal point, and gives the optimal rate of
decay available for this case. In this section we study the high frequency
behavior of $D_k-i(ka+ib)S_k$ for a smooth surface $\Gamma.$ In general one
cannot expect to have the sort of exact cancellation attained for a round
sphere, however we demonstrate the rather remarkable fact that for convex
bodies, the leading order part of the delay term is optimally canceled by
taking $a=1.$ It is less clear how to choose $b,$ but this is related to
insuring invertibility of $A(a,b;k)$ at zero frequency.

We begin by recalling the formul{\ae} for the relevant kernels:
\begin{equation}
  \begin{split}
    S_kf(x)&=\int\limits_{\Gamma}\frac{e^{ik|x-y|}}{4\pi|x-y|}f(y)dS_y\\
D_kf(x)&=\int\limits_{\Gamma}\frac{e^{ik|x-y|}\langle (x-y),\bn_y\rangle}{4\pi|x-y|^2}\left[\frac{1}{|x-y|}-ik\right]f(y)dS_y.
  \end{split}
\end{equation}
As is well known, these are weakly singular integral operators on a $\cC^1$
surface in $\bbR^3.$ 

Since we are interested in the high frequency asymptotics, the contribution of
the diagonal singularity is of less interest to us than the that of the other
critical points of the phase function
\begin{equation}
  \phi_x(y)=|x-y|.
\end{equation}
For completeness we state the contribution from the diagonal:
\begin{equation}
  \begin{split}
    S^{\diag}_kf(x)&=-\frac{f(x)}{2ik}+O(k^{-2})\\
D^{\diag}_kf(x)&=\frac{H(x)f(x)}{ik}+O(k^{-2}),
  \end{split}
\end{equation}
where $H(x)=\frac 12(\kappa_1(x)+\kappa_2(x))$ is the mean curvature of
$\Gamma$ at $x$ with respect to the outer normal vector $\bn_y.$ Observe that
the diagonal asymptotics of the combined field operator are given by:
\begin{equation}
  A^{\diag}(a,b;k)f(x)=\frac{1+a}{2}f(x)+(2H(x)-b)\frac{f(x)}{2ik}+O(k^{-2})
\end{equation}

Now we turn to the other critical points of $\phi_x:$
\begin{equation}\label{eqn62}
\begin{split}
  \cC_x&=\{y\in\Gamma:\nabla_y\phi_x(y)=\pm \bn_y\}\setminus\{x\}\\
&=\{y\in\Gamma:x=y\mp\phi_x(y) \bn_y\}\setminus\{x\}.
\end{split}
\end{equation}
Even for a strictly convex body, these critical points can be degenerate. For
example, if $\Gamma$ contains an open subset, $U,$ of a sphere, whose center
$x_0$ lies on $\Gamma,$ then $U\subset \cC_{x_0}.$ To use the standard
techniques of stationary phase we restrict our attention to the case of where
$\Gamma$ is $\cC^{\infty}$ and $\cC_x$ consists entirely of non-degenerate
critical points. There are, for example, many boundaries for which $\cC_x$
consists of a global maximum for every $x\in\Gamma.$ A simple analysis of this
property, using the characterization of $\cC_x$ in~\eqref{eqn62}, shows that it
is stable under $\cC^2$ small perturbations.

If $\hx$ is a non-degenerate critical point of $\phi_x,$ then $x$ lies along
the line $\{\hx+t\bn_{\hx}\}.$ We choose orthogonal coordinates so that $\hx=0,$
$\bn_{\hx}=(0,0,1),$ so that $x=(0,0,\pm d),$ with $d=|x-\hx|.$ For a convex body
we always have $-d.$ In these coordinates we represent the surface $\Gamma$ near to $0$ as a graph over
its tangent plane by taking
\begin{equation}
 y=(z,h(z))\text{ with }h(0)=\nabla h(0)=0.
\end{equation}
Let $\psi$ be a smooth function equal to 1 in a neighborhood of $0,$ with a
single critical point of $\phi_x$ in its support. 

With $x=(0,0,\pm d),$ the asymptotic contribution of $\hx$ to $S_kf(x),$ which
we denote by $ S_k^{\hx}f(x),$ is given by
\begin{equation}
  S_k^{\hx}f(x)\sim \int\frac{e^{ik(|z|^2\mp 2dh(z)+h(z)^2+d^2)^{\frac 12}}}
{4\pi(|z|^2\mp 2dh(z)+h(z)^2+d^2)^{\frac 12}}\frac{f(z)\psi(z)dz}{\sqrt{1+|\nabla h(z)|^2}}.
\end{equation}
To apply the stationary phase formula (at least to leading order) we need only
compute the Hessian of the phase function at $0;$ it is
\begin{equation}
  \cH_{x\hx}(0)=\frac{1}{d}\left(\begin{matrix} 1\mp dh_{11}(0)&\mp d h_{12}(0)\\
\mp dh_{21}(0)&1\mp d h_{22}(0)\end{matrix}\right).
\end{equation}
 According to stationary phase we see that
 \begin{equation}
  S_k^{\hx}f(x)\sim\frac{e^{\frac{\pi i\Sgn\cH_{x\hx}(0)}{4}}}{2
    k\sqrt{\det\cH_{x\hx}(0)}}
\frac{e^{ikd}f(\hx)}{d}+ O(k^{-2}).
 \end{equation}
 Recall that $\Sgn\cH_{x\hx}(0)$ is the signature of the Hessian, which is the
 number of positive eigenvalues minus the number of negative eigenvalues. For a
 maximum $\Sgn\cH_{x\hx}(0)=-2.$ Even for strictly convex bodies there can be
 several critical points with different signatures. If this happens then there
 are also points $x$ so that $\cC_x$ contains degenerate critical points.

The calculation for the double layer is similar:
\begin{multline}
  D_k^{\hx}f(x)\sim \int\frac{e^{ik(|z|^2\mp 2dh(z)+h(z)^2+d^2)^{\frac
        12}}(h(z)-\nabla h(z)\cdot z\pm d)f(z)\psi(z)}
{4\pi(|z|^2\mp 2dh(z)+h(z)^2+d^2)}\times\\
\left[\frac{1}{(|z|^2\mp
    2dh(z)+h(z)^2+d^2)^{\frac 12}}-ik\right]dz.
\end{multline}
 Once again, applying the stationary phase formula we see that
 \begin{equation}\label{eqn65}
  D_k^{\hx}f(x)\sim\frac{e^{\frac{\pi i\Sgn\cH_{x\hx}(0)}{4}}}{2
    \sqrt{\det\cH_{x\hx}(0)}}
\frac{\mp ie^{ikd}f(\hx)}{d}+ O(k^{-1}).
 \end{equation}

As noted above, if $\Gamma$ is convex, then $x=(0,0,-d),$ and therefore, we see
that
\begin{equation}
  D_k^{\hx}f(x)-ikS_x^{\hx}f(x)=O(k^{-1}),
\end{equation}
for all $x$ and non-degenerate critical points $\hx\in\cC_x.$ Unless each
$\cC_x$ consists of a single maximum, for every $x\in\Gamma,$ then, for some
points $x\in\Gamma,$ $\phi_x$ is sure to have degenerate critical points.  If
these are isolated and of finite degeneracy, then this difference is still
likely to be $o(1)$ as $k\to\infty.$ Of course if $\Gamma$ is not convex, then
it is perfectly possible that we may sometimes get the $-$ sign in the
asymptotics for $D_k^{\hx}f(x),$ equation~\eqref{eqn65}, so that we obtain
contributions of the form:
\begin{equation}
  D_k^{\hx}f(x)-iakS_x^{\hx}f(x)=\sim-i(1+a)\frac{e^{\frac{\pi i\Sgn\cH_{x\hx}(0)}{4}}}{
    \sqrt{\det\cH_{x\hx}(0)}}
\frac{e^{ik|x-\hx|}f(\hx)}{|x-\hx|}+ O(k^{-1}).
\end{equation}
Since $a>0$ is needed to keep poles out of the upper half plane, this may
prevent any source for the wave equation, defined by solving~\eqref{eqn3} with
an allowable choice of $(a,b),$ from decaying exponentially.

\frenchspacing {
  \bibliographystyle{siam}{\bibliography{alla-k}}}
\end{document}